\documentclass{article}
\usepackage{amssymb}
\usepackage{amsmath}
\usepackage{amsthm}

\usepackage{amsthm}

\newtheoremstyle{boldremark}% název stylu
  {\topsep}    % mezera nad
  {\topsep}    % mezera pod
  {\normalfont}  % font těla
  {}           % odsazení
  {\bfseries}  % font nadpisu
  {.}          % tečka po nadpisu
  {0.5em}      % mezera po nadpisu
  {}           % specifikace nadpisu

\theoremstyle{boldremark}
\usepackage{amsmath, amsthm}
\newtheorem{theorem}{Theorem}[section] 
\newtheorem{proposition}{Proposition}[section]
\newtheorem{lemma}[theorem]{Lemma}
\newtheorem{corollary}[theorem]{Corollary}

\newtheorem{remark}[theorem]{Remark}

\usepackage[english]{babel}

\usepackage[letterpaper,top=2cm,bottom=2cm,left=3cm,right=3cm,marginparwidth=1.75cm]{geometry}

\renewcommand\[{\begin{equation}}
\renewcommand\]{\end{equation}}
\allowdisplaybreaks[4]
\usepackage{amsmath}
\usepackage{comment}
\usepackage{graphicx}
\usepackage[colorlinks=true, allcolors=blue]{hyperref}

  \theoremstyle{definition}

  \numberwithin{equation}{section}

\begin{document}
\title{$H$-harmonic reproducing kernels on the ball}

\author{Matěj Moravík}
%\address{Mathematical Institute in Opava,\\
%Silesian University in Opava,\\
%Na Rybníčku 626/1,\\
%746 01 Opava,\\
%Czech Republic}
%\email{Mor0109@slu.cz}
\maketitle

\begin{abstract}
We consider the Szegő reproducing kernel associated with the space of $H$-harmonic functions on the unit ball in n-dimensional space, i.e. functions that are characterized by being annihilated by the hyperbolic Laplacian. This paper derives an explicit series expansion for the reproducing kernel in terms of a triple hypergeometric function introduced of Exton. Moreover, we demonstrate that the Szeg\H{o} kernel admits a representation as a finite sum of hypergeometric functions. We further show that the Szeg\H{o} kernel, for linearly dependent arguments, can be expressed in terms of the first Appell hypergeometric function. In addition we provide a series expansion for the weighted Bergman kernels.
\end{abstract}

\section{Introduction}
Let $B^n$ be the open unit ball in $\mathbb{R}^n$, $n > 2$. The orthogonal transformations
\begin{equation*}
    x \mapsto Ux \quad x \in \mathbb{R},\; U \in O(n),
\end{equation*}
maps $B^n$ and its boundary $\partial B^n$(unit sphere) onto themselves, and so do Moebius transformations
\begin{equation*}
    \phi_a(x) := \frac{a|x-a| + (1-|a|^2)(x-a)}{1 - 2\langle x, a \rangle + |x|^2|a|^2}
\end{equation*}
iterchanging origin $\textbf{0} \in B^n$ with point $a \in B^n$; here $\langle x,y\rangle$ denotes the usual scalar product.\\
The group $G$ generated by the $\phi_a,\;a \in B^n,$ and $U \in O(n)$ is called the Moebius group of $B^n$.
Let $\mu$ denote the normalized Lebesgue measure on $B^n$ \emph{i.e.} $\mu(B^n) = 1$, similarly $\sigma$ will denote the surface measure normalized so that $\sigma(\partial B^n) = 1.$ We define the Moebius-invariant measure $\rho$ on $B^n$ as follows
\begin{equation} \label{me1}
    d\rho(x) = \frac{d\mu(x)}{(1-|x|^2)^n}.
\end{equation} 
The unique hyperbolic metric (up to a constant factor) is then given by \begin{equation*}
    ds = \frac{|dx|}{1-|x|^2}.
\end{equation*}
The hyperbolic Laplacian is then defined as \begin{equation} \label{eq1}
    \Delta_hf(x) := \Delta(f \:\circ \: \phi_x0), \quad f\in C^2(B^n), \: x\in B^n;
\end{equation}
it follows from a direct calculation that  \begin{equation*}
    \Delta_h f(x) = (1-|x|^2)[(1- |x|^2)\Delta f(x) + 2(n-1)\langle x,\nabla f(x) \rangle ],
\end{equation*}
where $\Delta f =  \frac{\partial^2 f}{\partial x_1^2} + ...+ \frac{\partial^2 f}{\partial x_n^2}$ and $\nabla f = \big(\frac{\partial f}{\partial x_1} , ...,\frac{\partial f}{\partial x_n}\big)$.
Functions on $B^n$ annihilated by $\Delta_h$ are called hyperbolic harmonic, or $H$-harmonic for short. The vector space of all $H$-harmonic functions on the unit ball $B^n$ will be denoted by $\mathcal{H}.$ When $f \in \mathcal{H},$ then $f(x)$ equals its mean value with respect to the measure $\rho$ over any ball in $B^n$ with centre at $x$ (see~\cite[Corollary 4.1.3]{Stoll2016}).\\We can now consider the Hardy space \( \mathcal{H}^p_h(B^n) \) consisting of \( H \)-harmonic functions on \( B^n\) satisfying
$$
\|f\|_{\mathcal{H}^p} := \sup_{0 < r < 1} \left( \int_{\partial B^n} |f(r\eta)|^p \, d\sigma(\eta) \right)^{1/p} < \infty.
$$
It is well known that the space \( \mathcal{H}^2_h(B^n) \) is a reproducing kernel Hilbert space, thus we can consider its reproducing kernel, called the \textit{Szeg\H{o} kernel} $K_h(x, y)$, defined on $B^n \times B^n$, which is $H$-harmonic in both variables and satisfies the reproducing property
\begin{equation*}
    f(x) = \lim_{r\rightarrow1}\int_{\partial B^n} f(r\eta)\, K_h(x, r\eta)\, d\sigma(\eta), \quad \text{for all } x \in B^n \text{ and } f \in \mathcal{H}^2_h(B^n).
\end{equation*}
More generally, for any $s > -1$, one can consider $\textit{weighted $H$-harmonic Bergman space}$
\begin{equation*}
    \mathcal{H}_s := \{f \in L^2(B^n, d\rho_s), \quad f \text{ is $H$-harmonic} \},
\end{equation*}
where $$d\rho_s(x) = \frac{\Gamma(\frac{n}{2}+s+1)}{\pi^{\frac{n}{2}}\Gamma(s+1)}(1-|x|^2)^sd\mu(x),$$ and its reproducing kernel $K_h^{s}(x,y).$\\
Suppose that \( f \) belongs to $\mathcal{H}^p_h(B^n)$,
then  following representation holds for every \( x \in B^n \):
\[
f(x) = \lim_{r\rightarrow1}\int_{\partial B^n} f(r\eta)\, P_h(\eta, x)\, d\sigma(\eta),
\]
where \( P_h(\eta, x) \) denotes the hyperbolic Poisson kernel, defined on \( \partial B^n \times B^n \). It was shown in \cite[Chapter 5]{Stoll2016} that the Poisson kernel \(P_h(\eta, x)\) admits the explicit formula: 
$$
    P_h(\eta, x) = \frac{(1 - |x|^2)^{n-1}}{|x - \eta|^{2(n-1)}},
$$
for \(\eta \in \partial B^n\) and \(x \in B^n\).
In addition, it follows by a classical argument that the Szegő kernel \(K_h(x, y)\) is equal to 
\begin{equation} \label{poisson}
 K_h(x, y) = \int_{\partial B^n} P_h(\eta, x) P_h(\eta, y)\, d\sigma(\eta),
\end{equation}
for all \(x, y \in B^n\) (see~\cite[Theorem 2.1]{Stoll2019} for details).
Thus, \(K_h(x, y)\) can be viewed as an extension of the Poisson kernel from \(\partial B^n \times B^n\) to \(B^n \times B^n\).
The reproducing kernel $K_h(x,y)$ on the line $y = \lambda x$, specially in the case of \(\lambda = 1\), is of particular interest.
The function \(K_h(x,x)\) is of fundamental importance due to its role in estimating the pointwise values of functions in \(\mathcal{H}^2_h(B^n)\). To see this, for any \(f \in \mathcal{H}^2_h(B^n)\), we have:
\begin{equation}\label{78}
    |f(x)| = \left|\lim_{r\rightarrow1}\int_{\partial B^n} f(r\eta)\, K_h(x,r\eta)\, d\mu(\eta)\right|
    \leq \sqrt{K_h(x,x)}\, \|f\|_{\mathcal{H}^2_h(B^n)},
\end{equation}
where \eqref{78} follows from the Cauchy--Schwarz inequality and reproducing property of the Szegő kernel.
It was shown in (see~\cite[Theorem~3.1]{Stoll2019}) that in even dimensions $K_{h}(x,y)$ admits a closed-form representation, concretely:\\
\begin{equation*}
K_{h}(x, y) =
\frac{Q_n(x, y)}
     {\left(1 - 2\langle x, y \rangle + |x|^2|y|^2\right)^{\frac{3n}{2}-2}},
\end{equation*}
where $Q_n(x,y)$ is a symmetric polynomial.\\

In ~\cite[Theorem~5.3]{Stoll2019} it was also established that the Hilbert space $\mathcal{H}_s$ possesses a reproducing kernel that can be written as
$$
K_h^s(x, y) = \sum_{m=0}^\infty \frac{S_m(|x|^2)\, S_m(|y|^2)}{I_m(s)}\, Z_m(x, y),
$$
where the function $S_m(t)$ is defined by
$$
S_m(t) := \frac{(n-1)_m}{\left(\frac{n}{2}\right)_m} \, {}_2F_1\left(\begin{matrix} m,\, 1 - \frac{n}{2} \\ m + \frac{n}{2} \end{matrix} ;\, t \right),
$$
and $I_m(s)$ is given by
$$
I_m(s) := \frac{\Gamma\left(\frac{n}{2}+s+1\right) }{\Gamma\left(\frac{n}{2}\right) \Gamma(s + 1)} \int_0^1 t^{m + \frac{n}{2} - 1} (1 - t)^s S_m(t)^2 \, dt.
$$
The coefficients $I_m(s)$ were heavily studied in \cite{Ureyen2023}.\\
Here, \(Z_m(x, y)\) denotes the zonal harmonic of degree \(m\), and $(x)_m=x(x+1)\dots(x+m-1)$ is the usual Pochhammer symbol.\\

The main objective of this paper is to obtain an explicit formula for the \(H\)-harmonic Szeg\H{o} kernel \(K_h(x, y)\), expressed in terms of a certain hypergeometric series. Moreover, we show that \(K_h(x, y)\) can be written as a finite sum of hypergeometric functions.

In addition, we provide a representation of \(K_h(x, y)\) for the case of linearly dependent arguments, expressed in terms of the first Appell hypergeometric function.

Finally, we conclude by deriving the series expansion for the weighted Bergman kernel \(K_h^{s}(x, y)\).

Our main results are the following.
\medskip

\noindent
\textbf{Theorem} (\textit{Theorem~\ref{Pr1}}). Let $\alpha > 1$, $\beta > 1$  and $n \geq 2$ then for $x,\:y\in B^n$
\begin{equation}\label{prek}
    \int_{\partial B^n}\frac{1}{|x-\eta|^{2(\alpha-1)}}\frac{1}{|y-\eta|^{2(\beta-1)}} d\sigma(\eta) = \frac{1}{(1+|x|^2)^{\alpha-1}(1+|y|^2)^{\beta-1}}X_9(\alpha-1,\beta-1;\frac{n}{2};X,Y,Z),
\end{equation}
where 
\[X = \left(\frac{|x|}{1+|x|^2}\right)^2,\quad Y =\frac{2\langle x,y\rangle}{(1+|y|^2)(1+|x|^2)}, \quad Z = \left(\frac{|y|}{1+|y|^2}\right)^2
\]
and $X_9$ denotes Exton's triple hypergeometric series defined as
$$
X_9(\alpha, \beta;\, \gamma;\, x, y, z)
    := \sum_{p=0}^{\infty}\sum_{q=0}^{\infty}\sum_{l=0}^{\infty}
      \frac{(\alpha)_{2p+q}\, (\beta)_{2l+q}}
           {(\gamma)_{p+q+l}\; p! \; q! \; l!}
      \, x^p\, y^q\, z^l.
$$
\medskip
\noindent
\textbf{Corollary} (\textit{Corollary~\ref{cor:szego_kernel}}). 
For $x,y\in B^n$,
  \begin{equation}\label{eq:szego_kernel_explicit}
    K_h(x,y)
    = \frac{(1-\lvert x\rvert^2)^{n-1}(1-\lvert y\rvert^2)^{n-1}}
           {(1+\lvert x\rvert^2)^{n-1}(1+\lvert y\rvert^2)^{n-1}}
      X_9\!\bigl(n-1,n-1;\tfrac n2;X,Y,Z\bigr),
  \end{equation}
  where
$$
    X = \left(\frac{|x|}{1+|x|^2}\right)^2,\quad
    Y = \frac{2\langle x,y\rangle}{(1+\lvert x\rvert^2)\,(1+\lvert y\rvert^2)},\quad
    Z = \left(\frac{|y|}{1+|y|^2}\right)^2.
$$
\medskip
\noindent
\textbf{Theorem} (\textit{Theorem~\ref{theorem necooo}}). 
For \( x,\: y \in B^n \),
\begin{align}
K_h(x,y) 
&= 
\left(
    1 + 
    \frac{
        |x-y|^2
    }{
        (1 - |x|^{2})(1 - |y|^{2})
    }
\right)^{\!2 - \tfrac{3n}{2}}
\\[4pt]\notag
&\quad \times
\sum_{p,l=0}^{\,n-1}
\frac{(1-n)_{p+l}}{p!\,l!}\,
\frac{(n-1)_p\,(n-1)_l}{\left(\tfrac{n}{2}\right)_{p+l}}
\left( \frac{|x|^{2}}{|x|^{2}-1} \right)^{\!p}
\left( \frac{|y|^{2}}{|y|^{2}-1} \right)^{\!l}
\\[4pt]\notag
&\quad \times
{}_2F_1\!\left(
\begin{matrix}
1 + p - \tfrac{n}{2},\; 1 + l - \tfrac{n}{2} \\[4pt]
\tfrac{n}{2} + p + l
\end{matrix}
;\;
\frac{
    -|x-y|^2
}{
    (1 - |x|^{2})(1 - |y|^{2})
}
\right),
\end{align}
where ${}_2F_1$ is Gauss' hypergeometric function.
\medskip

\noindent
\textbf{Corollary} (\textit{Corollary~\ref{4.1}}). 
  Let \(x, \:y\in B^n\) with $y = \lambda x$ and \(0\leq \lambda \leq 1\). Then
  \[
    K_h(x,y)
    =\frac{(1-|x|^2)^{n-1}\,(1-|y|^2)^{n-1}}
          {(1+|x|)^{2(n-1)}\,(1+|y|)^{2(n-1)}}
     \,
     F_1\!\left(
\begin{matrix}
\frac{n - 1}{2};\; n - 1,\; n - 1 \\
n - 1
\end{matrix}
;\; \frac{4|x|}{(1 + |x|)^2},\; \frac{4|y|}{(1 + |y|)^2}
\right),
  \]
where $F_1$ is the first Appell hypergeometric function.
\medskip

\noindent

\noindent
\begin{comment}
\noindent
\textbf{Theorem} (\textit{Theorem~\ref{thm 4.3}}). 
  Let $n = 2k+1$ \(y, \:x\in B^n\) with $y = \lambda x$ and \(0\leq \lambda \leq 1\). Then
 \begin{equation}
K_{h,}(x, y)\\
= \frac{(1 - |x|^2)^{2k} (1 - |y|^2)^{2k}}{(1 + |x|)^{4k} (1 + |y|)^{4k}}\,
\frac{(n + t \partial_t)_n}{(n)_n}
\left[
\left(1 - \frac{4|x|t}{(1 + |x|)^2} \right)^{-2k}
\left(1 - \frac{4|y|t}{(1 + |y|)^2} \right)^{-2k}
\right]\Bigg|_{t = 1}.
\end{equation}
\medskip
\noindent
\end{comment}

\noindent
\textbf{Theorem} (\textit{Theorem~\ref{teylorbergman}}). 
  Let $x, y \in B^n$. Then
\begin{align}
K_h^{s}(x, y) = &\frac{(1-|x|^2)^{n-1}(1-|y|^2)^{n-1}}{(1+|x|^2)^{n-1}(1+|y|^2)^{n-1}}
\sum_{\gamma=0}^\infty\sum_{\alpha=0}^\infty\sum_{\beta=0}^\infty\mathcal{A_{\alpha, \,\beta, \,\gamma}}\:
X^{\alpha}\cdot
Y^{\gamma}\cdot
Z^{\beta},
\end{align}
 where
$$
    X = \left(\frac{\lvert x\rvert}{1+\lvert x\rvert^2}\right)^2,\quad
    Y = \frac{2\langle x,y\rangle}{(1+\lvert x\rvert^2)\,(1+\lvert y\rvert^2)},\quad
    Z = \left(\frac{\lvert y\rvert}{1+\lvert y\rvert^2}\right)^2
$$
and 
\begin{align*}
\mathcal{A_{\alpha, \,\beta, \,\gamma}} &:= \sum_{j=0}^{\min(\alpha,\beta)}
\frac{(-1)^j\big(\frac{n-2}{2}\big)_{\gamma+j}\,[n-2+2(\gamma+2j)]}
     {I_{\gamma+2j}(s)\,(n-2)\,j!\,\gamma!}\\
&\quad\times
\frac{(n-1)_{2\alpha+\gamma}(-\alpha)_j}{(\frac{n}{2})_{\alpha+\gamma+j}\,\alpha!}
\frac{(n-1)_{2\beta+\gamma}(-\beta)_j}{(\frac{n}{2})_{\beta+\gamma+j}\,\beta!}.
\end{align*}
\section{Notation and preliminaries}
Throughout this paper, $_2{}F_1$ will denote the (Gauss) hypergeometric function defined as 
$$
{}_2F_1\left(\begin{matrix} a,\, b \\[2pt] c \end{matrix};\, z \right)
= \sum_{m=0}^\infty \frac{(a)_m\, (b)_m}{(c)_m\, m!}\, z^m,
$$
and more generally, the generalized hypergeometric function $_pF_q$ will be denoted by
$$
{}_pF_q\left(\begin{matrix} a_1,\, a_2,\, \ldots,\, a_p \\[2pt] b_1,\, b_2,\, \ldots,\, b_q \end{matrix};\, z \right)
= \sum_{m=0}^\infty \frac{(a_1)_m\, (a_2)_m\, \cdots\, (a_p)_m}{(b_1)_m\, (b_2)_m\, \cdots\, (b_q)_m\, m!}\, z^m,
$$
where $a_1, \ldots, a_p$, $b_1, \ldots, b_q$, and $z$ are real numbers, with the restriction that none of the $b_j$ are zero or negative integers. The 
Pochhammer symbol $(a)_m$ is defined as
$$
(a)_m = a(a+1)\cdots(a+m-1) = \frac{\Gamma(a+m)}{\Gamma(a)}.
$$
We will also encounter two-variable analogues of hypergeometric functions. In particular, we will require the notion of the first Appell hypergeometric function:
$$
F_1\left(\begin{matrix} a;\, b_1,\, b_2 \\ c \end{matrix};\, x,\, y \right)
= \sum_{m=0}^\infty\sum_{l=0}^\infty \frac{(a)_{m+l}\, (b_1)_m\, (b_2)_l}{(c)_{m+l}\, m!\, l!}\, x^m\, y^l,
$$
convergent absolutely for $\max\{|x|,\,|y|\} < 1.$ \\
Our notation largely follows standard conventions, in particular those found in~\cite{BatemanHTF}.\\
We will also require the notion of Exton's triple hypergeometric function, denoted in Exton's own terminology by $X_9$, which we had already mentioned in the previous section. It is defined by the following triple series:
$$
X_9(\alpha, \beta;\, \gamma;\, x, y, z)
    := \sum_{p=0}^{\infty}\sum_{q=0}^{\infty}\sum_{l=0}^{\infty}
      \frac{(\alpha)_{2p+q}\, (\beta)_{2l+q}}
           {(\gamma)_{p+q+l}\; p! \; q! \; l!}
      \, x^p\, y^q\, z^l,
$$
which converges absolutely for
$|x| < \tfrac{1}{4}$, $|z| < \tfrac{1}{4}$, and $|y| < \tfrac{1}{2} + \sqrt{(1 - 4|x|)(1 - 4|z|)}$.
See~\cite{Exton1982} for more details.
We shall appeal to various properties of hypergeometric functions as we proceed.

\begin{comment}
 Here we are going to state the following contiguous lemma.
  \begin{lemma} \label{conting}
    Let $n \in \mathbb{N}$. Then
    \begin{equation}\label{pom_rovnice}
        F_1\!\left( \begin{matrix} \alpha;\; \beta,\; \gamma \\ n \end{matrix} ;\; x, y \right) = \frac{(n+t\partial_t)_n}{(n)_n} F_1\!\left( \begin{matrix} \alpha;\; \beta,\; \gamma \\ 2n \end{matrix} ;\; xt, yt \right)\Bigg|_{t=1}.
    \end{equation}
    Here $\partial_t$ is shorthand for $\frac{\partial}{\partial t}$.
\end{lemma}
\begin{proof}
    Note that $t\partial_tt^k = kt^{k},$ therefore we have:
    \begin{equation*}
        (n + t\partial_t)_nt^k = \frac{\Gamma(2n+k)}{\Gamma(n+k)}t^k,
    \end{equation*}
    and as $\Gamma(2n) = \Gamma(n)(n)_n$, claim follows.
\end{proof}  
\end{comment}
 Zonal harmonics are homogeneous, symmetric,  polynomials of total degree \(m\). Moreover, \(Z_m(x, y)\) serves as the reproducing kernel for the space of spherical harmonics of degree \(m\), \emph{i.e.}, the restrictions to the sphere \(\partial{B}^n\) of homogeneous harmonic polynomials in \(\mathbb{R}^n\). That is, for any  \(p \in \mathcal{H}_m(\partial{B}^n)\), we have
$$
p(x) = \int_{\partial{B}^{n}} Z_m(x, \eta)\, {p}(\eta)\, d\sigma(\eta),
$$
where \(\mathcal{H}_m(\partial{B}^n)\) denotes the space of spherical harmonics of degree \(m\). For more information on $Z_m(x,y)$ see~\cite[Chapter 5]{AxlerBourdonRamey}.

When both $x$ and $y$ lie on the unit sphere, the zonal harmonics are given by the Gegenbauer polynomials evaluated at the inner product:
\begin{equation} \label{zonal}
    Z_m(x, y) = \frac{m+\lambda}{\lambda}C_m^{\lambda}(\langle x, y\rangle), 
\end{equation}
where $C_m^{\lambda}$ is the Gegenbauer polynomial of degree $m$ and parameter $\lambda = \frac{n-2}{2}$. Explicitly, the Gegenbauer polynomials are given by
$$
C_m^{\lambda}(t) = \sum_{k=0}^{\lfloor m/2 \rfloor} \frac{(-1)^k\, \Gamma(m - k + \lambda)}{k!\, (m-2k)!\, \Gamma(\lambda)}\, (2t)^{m-2k},
$$
see~\cite[Chapter 4]{SteinWeiss1971}.
\section{Series expansion of the Szeg\"o kernel}
\begin{theorem} \label{Pr1}
For each $\alpha > 1$, $\beta > 1$  and $n \geq 2$ we have
\begin{equation}\label{prek}
    \int_{\partial B^n}\frac{1}{|x-\eta|^{2(\alpha-1)}}\frac{1}{|y-\eta|^{2(\beta-1)}}\,d\sigma(\eta) = \frac{1}{(1+|x|^2)^{\alpha-1}(1+|y|^2)^{\beta-1}}\,X_9(\alpha-1,\beta-1;\frac{n}{2};X,Y,Z),
\end{equation}
where $X = \left(\frac{|x|}{1+|x|^2}\right)^2$, $Y =\frac{2\langle x,y\rangle}{(1+|y|^2)(1+|x|^2)}$ and $Z = \left(\frac{|y|}{1+|y|^2}\right)^2$.
\end{theorem}
\begin{proof}
   Since the unit sphere \( \partial B^n \) is invariant under the action of the orthogonal group \( O(n) \), we may choose an orthogonal transformation \( R \in O(n) \) such that
$$
Rx = (a, 0, \dots, 0), \qquad Ry = (b, c, 0, \dots, 0)
$$
where $b = \frac{\langle x,y\rangle}{|x|}$, $c = \sqrt{|y|^2 - |b|^2}$ and $a = |x|$. For $x=0$, we set $b=0$. The integral over the sphere is invariant under such a change of variables, and hence we may assume without loss of generality that \( x \) and \( y \) are of this simplified form. Our integral therefore becomes:
\begin{equation} \label{eq7}
    \int_{\partial B^n}\frac{1}{(1-2a\eta_{1} + |x|^2)^{\alpha-1}}\frac{1}{(1-2b\eta_{1} - 2c\eta_{2}  + |y|^2)^{\beta-1}}d\sigma(\eta).
\end{equation}
Using the fact, that $\frac{1}{(1+x)^\alpha} = \frac{1}{\Gamma(\alpha)}\int_{0}^{\infty} e^{-t(1+x)}t^{\alpha -1}dt$, valid for $\alpha > 0$, we can rewrite each fraction as its Laplace transform, thus \eqref{eq7} becomes:
\begin{equation*}
    \frac{1}{\Gamma(\alpha-1)\Gamma(\beta-1)}\int_{\partial B^n}
\int_0^\infty \int_0^\infty 
t^{\alpha - 2} s^{\beta - 2} 
e^{- t(1 - 2a \eta_1 + |x|^2) 
- s(1 - 2b \eta_1 - 2c \eta_2 + |y|^2)} 
\, dt \, ds \, d\sigma(\eta).
\end{equation*}
Changing the order of integration (justified by absolute convergence of whole integral) and regrouping terms we arrive at:
\begin{equation} \label{jadro0}
    \frac{1}{\Gamma(\alpha - 1)\Gamma(\beta - 1)} 
\int_0^\infty \int_0^\infty 
t^{\alpha - 2} s^{\beta - 2} 
e^{- t(1 + |x|^2) - s(1 + |y|^2)} 
\left( 
\int_{\partial B^n} 
e^{2(a t + b s)\eta_1 + 2c s \eta_2} 
\, d\sigma(\eta) 
\right)
\, dt \, ds.
\end{equation}
Once more we use use an orthogonal transformation to align the vector $$
v := (2a t + 2b s, 2c s, 0, \dots, 0) \in \mathbb{R}^n
$$
with the first coordinate axis. More precisely, there exists a rotation \( S \in O(n)\) such that
$$
S v = (2\sqrt{(a t + b s)^2 + (c s)^2}, 0, \dots, 0).
$$
Again the spherical measure \( d\sigma \) is invariant under rotations, thus we have
\begin{equation} \label{baseel integral}
     \int_{\partial B^n} e^{\langle v, \eta \rangle} \, d\sigma(\eta)= \int_{\partial B^n} e^{\sqrt{(2a t + 2b s)^2 + (2c s)^2} \cdot \eta_1} \, d\sigma(\eta).
\end{equation}
Making a change to spherical coordinates, we have $\eta_1 = \cos(\theta_1)$ and the surface element is
$$
d\sigma(\eta) = (\omega_{n-1})^{-1}
\sin^{n-2}\theta_1 \sin^{n-3}\theta_2 \cdots \sin\theta_{n-2}\,
d\theta_1 \cdots d\theta_{n-1},
$$
where $\omega_{n-1}$ is surface area of the $n$ dimensional sphere.
Thus,  \eqref{baseel integral} can be written as
$$
\int_{\partial B^n} e^{\langle v, \eta \rangle} d\sigma(\eta)
= \int_{0}^{2\pi} \int_{0}^{\pi} \cdots \int_{0}^{\pi}
e^{|v| \cos\theta_1}
\prod_{j=1}^{n-2} \sin^{n-1-j}\theta_j \,
d\theta_1 \cdots d\theta_{n-1} \\
= \frac{\omega_{n-2}}{\omega_{n-1}} \int_0^\pi e^{|v| \cos\theta_1} \sin^{n-2}\theta_1\, d\theta_1.
$$
It follows from \cite[Eq.~3.915.4]{GradshteynRyzhik} that 
$$
\int_0^\pi e^{|v| \cos\theta_1} \sin^{n-2}\theta_1\, d\theta_1
= \sqrt{\pi} \, \Gamma\left( \frac{n-1}{2} \right)
  \left( \frac{2}{|v|} \right)^{(n-2)/2}
  I_{(n-2)/2}(|v|),
$$
where \( I_\nu(z) \) is the modified Bessel function of the first kind. From \cite[Eq. 8.445]{{GradshteynRyzhik}} we know that $I_\nu(z)$ admits series expansion:
$$
I_{(n-2)/2}(z) = \frac{ \left( \frac{z}{2} \right)^\frac{(n-2)}{2} }{ \Gamma(\frac{n}{2}) } \;
{}_0F_1\left( \frac{n}{2};\; \frac{z^2}{4} \right).
$$
Now, as $\frac{\omega_{n-2}}{\omega_{n-1}}
= \frac{ \Gamma\left( \frac{n}{2} \right) }
       { \sqrt{\pi}\; \Gamma\left( \frac{n-1}{2} \right) }$ we get, that \eqref{baseel integral} is equal to 
       
\begin{equation} \label{jadro1}
\int_{\partial B^n} e^{\sqrt{(2a t + 2b s)^2 + (2c s)^2} \cdot \eta_1} \, d\sigma(\eta) ={}_0F_1\left( \frac{n}{2}; (a t + b s)^2 + (c s)^2 \right).
\end{equation}
Remembering that $a = |x|$, $b = \frac{\langle x,y\rangle}{|x|}$ and $c = \sqrt{|y|^2 - |b|^2}$ we have that $((a t + b s)^2 + (c s)^2)$ is equal to $(t|x|)^2 + 2\langle x,y\rangle t s + (s|y|)^2$. Putting this back into \eqref{jadro0} and expanding ${}_0F_1$ we see that \eqref{jadro0} is equal to
\begin{align*}
& \int_0^\infty \int_0^\infty 
t^{\alpha - 2} s^{\beta - 2} 
e^{- t(1 + |x|^2) - s(1 + |y|^2)} 
\sum_{k=0}^\infty \frac{1}{\left(\tfrac{n}{2}\right)_k} \cdot \frac{((t|x|)^2 + 2\langle x,y\rangle t s + (s|y|)^2)^k}{k!} \, dt \, ds \\
&= \int_0^\infty \int_0^\infty 
t^{\alpha - 2} s^{\beta - 2} 
e^{- t(1 + |x|^2) - s(1 + |y|^2)} 
\sum_{k=0}^\infty \frac{1}{\left(\tfrac{n}{2}\right)_k k!} 
\sum_{\substack{p, q, l \geq 0 \\ p + q + l = k}} 
\binom{k}{p, q, l} 
(t|x|)^{2p} \cdot (2 \langle x, y \rangle t s)^q \cdot (s|y|)^{2l} 
\, dt \, ds,
\end{align*}
where \( \binom{k}{p,q,l} = \frac{k!}{p! q! l!} \) is the multinomial coefficient.

We now interchange the order of summation and integration, which is again justified by absolute convergence of the integrals for \( \alpha, \beta > 1 \). This yields
$$
\sum_{k=0}^\infty \frac{1}{\left(\tfrac{n}{2}\right)_k k!} 
\sum_{\substack{p, q, l \geq 0 \\ p + q + l = k}} 
\binom{k}{p,q,l} \cdot 
(2 \langle x, y \rangle)^q \cdot |x|^{2p} \cdot |y|^{2l} 
\left( 
\int_0^\infty e^{-t(1 + |x|^2)} t^{2p + q + \alpha - 2} \, dt
\right)
\left(
\int_0^\infty e^{-s(1 + |y|^2)} s^{2l + q + \beta - 2} \, ds
\right).
$$
Using the identity
$\int_0^\infty e^{-\lambda t} t^{\mu - 1} \, dt = \frac{\Gamma(\mu)}{\lambda^\mu}$ we obtain
$$
\int_0^\infty e^{-t(1 + |x|^2)} t^{2p + q + \alpha - 2} \, dt = \frac{\Gamma(2p + q + \alpha - 1)}{(1 + |x|^2)^{2p + q + \alpha - 1}},
$$
and similarly for the \(s\)-integral. Inserting these expressions back into the sum, using the definition of multinational coefficients, grouping terms by their powers, and using the identities $\Gamma(2p + q + \alpha - 1) = (\alpha -1)_{2p+q}\Gamma(\alpha-1)$, $\Gamma(2p + q + \beta - 1) = (\beta -1)_{2p+q}\Gamma(\beta-1)$, we obtain 
\begin{multline*}
\frac{\Gamma(\alpha-1)\Gamma(\beta-1)}{(1+|x|^{2})^{\alpha-1}(1+|y|^{2})^{\alpha-1}} 
\sum_{k=0}^\infty \frac{1}{\left(\tfrac{n}{2}\right)_k k!} 
\sum_{\substack{p, q, l \geq 0 \\ p + q + l = k}} 
\frac{k!(\alpha-1)_{2p+q}(\beta-1)_{2l+q}}{p!q!l!} \\ \times
\left( \frac{|x|}{1+|x|^2} \right)^{2p} 
\left( \frac{2 \langle x, y \rangle}{(1+|x|^{2})(1+|y|^{2})} \right)^q 
\left( \frac{|y|}{1+|y|^2} \right)^{2l}.
\end{multline*}
After cancelling $k!$ and the omitted gamma factors, followed by a change of the summation indices, we finally arrive at:
\begin{equation*}
    \frac{1}{(1+|x|^{2})^{\alpha-1}(1+|y|^{2})^{\beta-1}} 
\sum_{p,q,l}^\infty \frac{(\alpha-1)_{2p+q}(\beta-1)_{2l+q}}{\left(\tfrac{n}{2}\right)_{p+q+l}\:p!\,q!\,l! } 
\frac{}{} 
\left( \frac{|x|}{1+|x|^2} \right)^{2p} 
\left( \frac{2 \langle x, y \rangle}{(1+|x|^{2})(1+|y|^{2})} \right)^q 
\left( \frac{|y|}{1+|y|^2} \right)^{2l},
\end{equation*}
which is exactly what we wanted to show.
\end{proof}

\begin{corollary}\label{cor:szego_kernel}
  Let $x,y\in B^n$.  The Szegö kernel $K_h(x,y)$ can be written as
  \begin{equation*}\label{eq:szego_kernel_explicit}
    K_h(x,y)
    = \frac{(1-\lvert x\rvert^2)^{n-1}(1-\lvert y\rvert^2)^{n-1}}
           {(1+\lvert x\rvert^2)^{n-1}(1+\lvert y\rvert^2)^{n-1}}
      X_9\!\bigl(n-1,n-1;\tfrac n2;X,Y,Z\bigr),
  \end{equation*}
  where
$$
    X = \left(\frac{\lvert x\rvert}{1+\lvert x\rvert^2}\right)^2,\quad
    Y = \frac{2\langle x,y\rangle}{(1+\lvert x\rvert^2)\,(1+\lvert y\rvert^2)},\quad
    Z = \left(\frac{\lvert y\rvert}{1+\lvert y\rvert^2}\right)^2.
$$
\end{corollary}

\begin{proof}
  The result follows directly from
  Proposition~\ref{Pr1} and the equation~\eqref{poisson}.
\end{proof}
\begin{theorem} \label{theorem necooo}
    Let $x,y\in B^n$, then the Szegö kernel $K_h(x,y)$ can be written as
\begin{align}\notag
K_h(x,y) 
&= 
\left(
    1 +
    \frac{
        |x-y|^2
    }{
        (1 - |x|^{2})(1 - |y|^{2})
    }
\right)^{\!2 - \tfrac{3n}{2}}
\\[4pt]\notag
&\quad \times
\sum_{p,l=0}^{\,n-1}
\frac{(1-n)_{p+l}}{p!\,l!}\,
\frac{(n-1)_p\,(n-1)_l}{\left(\tfrac{n}{2}\right)_{p+l}}
\left( \frac{|x|^{2}}{|x|^{2}-1} \right)^{\!p}
\left( \frac{|y|^{2}}{|y|^{2}-1} \right)^{\!l}
\\[4pt]\label{szego_roz}
&\quad \times
{}_2F_1\!\left(
\begin{matrix}
1 + p - \tfrac{n}{2},\; 1 + l - \tfrac{n}{2} \\[4pt]
\tfrac{n}{2} + p + l
\end{matrix}
;\;
\frac{
    -|x-y|^2
}{
    (1 - |x|^{2})(1 - |y|^{2})
}
\right)
\end{align}

\end{theorem}
\begin{proof}
From Corollary~\ref{cor:szego_kernel} it follows that
\begin{equation*}
    K_h(x,y)
    = 
    \frac{(1 - \lvert x \rvert^2)^{\,n - 1}\,(1 - \lvert y \rvert^2)^{\,n - 1}}
         {(1 + \lvert x \rvert^2)^{\,n - 1}\,(1 + \lvert y \rvert^2)^{\,n - 1}}
    \;
    X_9\!\left( n - 1,\, n - 1;\, \tfrac{n}{2};\, X,\, Y,\, Z \right),
\end{equation*} 
with
\begin{equation*}
    X = \left( \frac{\lvert x \rvert}{\,1 + \lvert x \rvert^2\,} \right)^{\!2},
    \qquad
    Y = \frac{2 \langle x, y \rangle}
            {(1 + \lvert x \rvert^2)\,(1 + \lvert y \rvert^2)},
    \qquad
    Z = \left( \frac{\lvert y \rvert}{\,1 + \lvert y \rvert^2\,} \right)^{\!2},
\end{equation*}
where
\begin{equation*}
    X_9\!\left( n - 1,\, n - 1;\, \tfrac{n}{2};\, X,\, Y,\, Z \right) 
    = 
    \sum_{p, q, l = 0}^{\infty} 
    \frac{
        (n - 1)_{2l + q}\,(n - 1)_{2p + q}
    }{
        \left( \tfrac{n}{2} \right)_{p + q + l}\, p!\, q!\, l!
    } 
    X^{p}\, Y^{q}\, Z^{l}.
\end{equation*}
For brevity we will denote
\(
X_9\!\left(n - 1,\, n - 1;\, \tfrac{n}{2};\, X, Y, Z\right)
\)
simply by \(X_9\).\\
Observing that
$$
(n - 1)_{2p + q} 
= (n - 1)_q \left( \tfrac{n - 1 + q}{2} \right)_p 
                  \left( \tfrac{n + q}{2} \right)_p 4^p,
$$
the inner sum over \(p\) can be recognized as a Gauss hypergeometric series, in particular,
\begin{align}\notag
X_9 
&= \sum_{p, q, l = 0}^\infty 
      \frac{(n - 1)_{2l + q}\,(n - 1)_{2p + q}}
           {\left( \tfrac{n}{2} \right)_{p + q + l}\, p!\, q!\, l! } 
      X^{p}\, Y^{q}\, Z^{l} 
\\[8pt]\notag
&= \sum_{p, q, l = 0}^\infty  
    \frac{(n - 1)_{2l + q}\,(n - 1)_q
          \left( \tfrac{n - 1 + q}{2} \right)_p
          \left( \tfrac{n + q}{2} \right)_p}
         {\left( \tfrac{n}{2} \right)_{q + l}\,
          \left( \tfrac{n}{2} + q + l \right)_{p}\,
          p! \, l! \, q!}\,
    (4X)^{p}\, Y^{q}\, Z^{l} 
\\[8pt]\label{Stoya}
&= \sum_{q, l = 0}^\infty 
    \frac{(n - 1)_{2l + q}\,(n - 1)_{q}}
         {\left( \tfrac{n}{2} \right)_{q + l} \, l! \, q!}\,
    Y^{q}\, Z^{l}\;
{}_2F_1\!\left(
        \begin{matrix}
            \tfrac{n - 1 + q}{2},\; \tfrac{n + q}{2} \\[6pt]
            \tfrac{n}{2} + q + l
        \end{matrix}
        ;\; \dfrac{4|x|^2}{(1 + |x|^2)^{2}}
    \right).
\end{align}
Now, by applying Gauss’ quadratic transformation (see~\cite[p.~64]{BatemanHTF}),
\begin{equation}\label{Gauss}
    {}_2F_1\!\left(
    \begin{matrix}
        \tfrac{1}{2}a,\; \tfrac{1}{2}a + \tfrac{1}{2} \\[2pt]
        a - b + 1
    \end{matrix}
    ;\,
    \dfrac{4z}{(1 + z)^2}
    \right)
    = (1 + z)^{a}\,
    {}_2F_1\!\left(
    \begin{matrix}
        a,\; b \\[2pt]
        a - b + 1
    \end{matrix}
    ;\,
    z
    \right),
\end{equation}
we obtain
\begin{equation}\label{pomocc}
    {}_2F_1\!\left(
    \begin{matrix}
        \tfrac{n - 1 + q}{2},\; \tfrac{n + q}{2} \\[4pt]
        \tfrac{n}{2} + q + l
    \end{matrix}
    ;\;
    \dfrac{4|x|^{2}}{(1 + |x|^{2})^{2}}
    \right)
    = (1 + |x|^{2})^{\,n - 1 + q}\,
    {}_2F_1\!\left(
    \begin{matrix}
        n - 1 + q,\; \tfrac{n}{2} - l \\[4pt]
        \tfrac{n}{2} + q + l
    \end{matrix}
    ;\;
    |x|^{2}
    \right).
\end{equation}
Substituting \eqref{pomocc} into \eqref{Stoya} and simplifying the Pochhammer symbols, we have  
\begin{align}\notag
    X_9 &= (1 + |x|^2)^{\,n - 1} 
    \sum_{p, q, l = 0}^\infty 
    \frac{(n - 1)_{2l}\,(n - 1 + 2l)_q \left( \tfrac{n}{2} - l \right)_p (n - 1)_{q + p}}
         {\left( \tfrac{n}{2} \right)_{l}\, \left( \tfrac{n}{2} + l \right)_{q + p}\, p!\, q!\, l!}\,
    |x|^{2p}\!
    \left( \frac{2 \langle x, y \rangle}{1 + |y|^2} \right)^{q} 
    Y^{l} 
\\[6pt]\label{nevim}
    &= (1 + |x|^2)^{\,n - 1} 
    \sum_{l = 0}^\infty 
    \frac{(n - 1)_{2l}}{\left( \tfrac{n}{2} \right)_l\, l!}\,
    Y^{l}\,
    F_1\!\left(
        \begin{matrix}
            n - 1;\; \tfrac{n}{2} - l,\; n - 1 + 2l \\[4pt]
            \tfrac{n}{2} + l
        \end{matrix}
        ;\;
        |x|^2,\;
        \dfrac{2 \langle x, y \rangle}{\,1 + |y|^2\,}
    \right),
\end{align}
where \(F_1\) denotes the first Appell hypergeometric function.\\
Since \(F_1\) satisfies the following \emph{Euler-type transformation} (see~\cite[p.~239, Eq.~(2)]{BatemanHTF}):
\begin{equation}\label{eq:F1-transform}
F_{1}\!\left(
\begin{matrix}
\alpha;\; \beta,\, \beta' \\[4pt]
\gamma
\end{matrix}
;\; x,\, y
\right)
=
(1 - x)^{-\alpha}\,
F_{1}\!\left(
\begin{matrix}
\alpha;\; \gamma - \beta - \beta',\,\beta'  \\[4pt]
\gamma
\end{matrix}
;\;
\dfrac{x}{\,x - 1\,},\;
\dfrac{x - y}{\,x - 1\,}
\right),
\end{equation}
we obtain
\begin{align}\notag
&F_{1}\!\left(
\begin{matrix}
n - 1;\; \tfrac{n}{2} - l,\, n - 1 + 2l \\[2pt]
\tfrac{n}{2} + l
\end{matrix}
;\; |x|^{2},\, \dfrac{2\langle x, y\rangle}{1 + |y|^{2}}
\right)
\\[6pt]\label{eq:F1-transform-applied}
&\quad =
(1 - |x|^{2})^{-(n - 1)}\,
F_{1}\!\left(
\begin{matrix}
n - 1;\; 1 - n,\, n - 1 + 2l \\[2pt]
\tfrac{n}{2} + l
\end{matrix}
;\;
\dfrac{|x|^{2}}{|x|^{2} - 1},\;
\dfrac{|x|^{2}(1 + |y|^{2}) - 2\langle x, y\rangle}
      {(1 + |y|^{2})(|x|^{2} - 1)}
\right).
\end{align}
After substituting \eqref{eq:F1-transform-applied} into \eqref{nevim} and simplifying the Pochhammer symbols, we derive
\begin{align}\label{kon}
    X_9 &= \left( \frac{1 + |x|^{2}}{1 - |x|^{2}} \right)^{n - 1}
    \sum_{p, q, l = 0}^\infty
    \frac{(n - 1)_{2l + q}\,(n - 1)_{p + q}\,(1 - n)_p}
         {\left( \tfrac{n}{2} \right)_{p + q + l}\, p!\, q!\, l!}\,
    \left( \frac{|x|^{2}}{|x|^{2} - 1} \right)^{p}
    S^{q}\, Y^{l} 
\\[6pt]\notag
    &= \left( \frac{1 + |x|^{2}}{1 - |x|^{2}} \right)^{n - 1}
    \sum_{p, q = 0}^\infty
    \frac{(n - 1)_{q}\,(n - 1)_{p + q}\,(1 - n)_p}
         {\left( \tfrac{n}{2} \right)_{p + q}\, p!\, q!}\,
    \left( \frac{|x|^{2}}{|x|^{2} - 1} \right)^{p}
    S^{q}
\\[6pt]\notag
    &\quad \times
    {}_2F_1\!\left(
        \begin{matrix}
            \tfrac{n - 1 + q}{2},\; \tfrac{n + q}{2} \\[6pt]
            \tfrac{n}{2} + q + p
        \end{matrix}
        ;\;
        \dfrac{4 |y|^{2}}{(1 + |y|^{2})^{2}}
    \right),
\end{align}
where
$$
S := \dfrac{|x|^{2}(1 + |y|^{2}) - 2\langle x, y \rangle}
          {(1 + |y|^{2})(|x|^{2} - 1)}.
$$
Again, by applying \eqref{Gauss}, we may rewrite
\begin{equation*}\label{eq:Gauss-applied-y}
{}_2F_1\!\left(
\begin{matrix}
\tfrac{n - 1 + q}{2},\; \tfrac{n + q}{2} \\[4pt]
\tfrac{n}{2} + q + p
\end{matrix}
;\;
\dfrac{4 |y|^{2}}{(1 + |y|^{2})^{2}}
\right)
=
(1 + |y|^{2})^{\,n - 1 + q}\,
{}_2F_1\!\left(
\begin{matrix}
n - 1 + q,\; \tfrac{n}{2} - p \\[4pt]
\tfrac{n}{2} + q + p
\end{matrix}
;\;
|y|^{2}
\right).
\end{equation*}
Thus
\begin{align}\label{konec1}
    X_9 &= \left(\frac{(1+|x|^2)(1+|y|^2)}{1-|x|^2}\right)^{n-1}\sum_{p,q,l}^\infty\frac{(n-1)_{l+q}(\frac{n}{2}-p)_l\,(n-1)_{p+q}\,(1-n)_p}{(\frac{n}{2})_{p+q+l}\,p!\,q!\,l!}\left(\frac{|x|^2}{|x|^2-1}\right)^p \, \left((1+|y|^2)S\right)^q \, |y|^{2l}\\\notag
    &= \left(\frac{(1+|x|^2)(1+|y|^2)}{1-|x|^2}\right)^{n-1} \sum_{p=0}^\infty \frac{(1-n)_p(n-1)_p}{(\frac{n}{2})_p\,p!}\left(\frac{|x|^2}{|x|^2-1}\right)^p\\[4 pt]\notag
    &\qquad \times F_{1}\!\left(
\begin{matrix}
n-1;\; \tfrac{n}{2}-p,\, n-1+p \\[2pt]
\tfrac{n}{2}+p
\end{matrix}
;\; |y|^{2},\, \dfrac{|x|^{2}(1+|y|^{2})-2\langle x,y\rangle}{(|x|^{2}-1)}
\right)\label{konec}
    \end{align}
Applying \eqref{eq:F1-transform} once more, we arrive at
\begin{align*}
&F_{1}\!\left(
\begin{matrix}
n - 1;\; \tfrac{n}{2} - p,\, n - 1 + p \\[2pt]
\tfrac{n}{2} + p
\end{matrix}
;\; |y|^{2},\,
\dfrac{|x|^{2}(1 + |y|^{2}) - 2\langle x, y \rangle}{\,|x|^{2} - 1\,}
\right)
\\[6pt]
&\quad =
(1 - |y|^{2})^{-(n - 1)}\,
F_{1}\!\left(
\begin{matrix}
n - 1;\; 1 - n + p,\, n - 1 + p \\[2pt]
\tfrac{n}{2} + p
\end{matrix}
;\;
\dfrac{|y|^{2}}{|y|^{2} - 1},\;
\dfrac{\,2\langle x, y \rangle - |x|^{2} - |y|^{2}\,}
     {\,(|x|^{2} - 1)(|y|^{2} - 1)\,}
\right).
\end{align*}
Hence
\begin{align}\notag
X_9 
&= 
\left[
    \frac{(1 + |x|^{2})(1 + |y|^{2})}
         {(1 - |x|^{2})(1 - |y|^{2})}
\right]^{\!n - 1}
\sum_{p, q, l = 0}^\infty 
\frac{
    (n - 1)_{l + q}\,(n - 1)_{p + q}\,(1 - n)_{p + l}
}{
    \left( \tfrac{n}{2} \right)_{p + q + l}\, p!\, q!\, l!
}
\left( \frac{|x|^{2}}{|x|^{2}-1} \right)^{\!p}
\left( \frac{|y|^{2}}{|y|^{2}-1} \right)^{\!l}
\\[6pt]\notag
&\quad \times 
\left[
    \frac{
        2\langle x, y \rangle - |x|^{2} - |y|^{2}
    }{
        (|x|^{2} - 1)(|y|^{2} - 1)
    }
\right]^{\!q}
\\[10pt]\notag
&= 
\left[
    \frac{(1 + |x|^{2})(1 + |y|^{2})}
         {(1 - |x|^{2})(1 - |y|^{2})}
\right]^{\!n - 1}
\sum_{p, l = 0}^{n-1}
\frac{
    (n - 1)_{p}\,(n-1)_{l}\,(1 - n)_{p + l}
}{
    \left( \tfrac{n}{2} \right)_{p + l}\, p!\, l!
}
\left( \frac{|x|^{2}}{|x|^{2}-1} \right)^{\!p}
\left( \frac{|y|^{2}}{|y|^{2}-1} \right)^{\!l}
\\[8pt]\notag
&\quad \times 
{}_2F_1\!\left(
\begin{matrix}
n - 1 + p,\; n - 1 + l \\[4pt]
\dfrac{n}{2} + p + l
\end{matrix}
;\;
\frac{
    2\langle x, y \rangle - |x|^{2} - |y|^{2}
}{
    (|x|^{2} - 1)(|y|^{2} - 1)
}
\right).\\\notag
\end{align}
Multiplying by 
$$
\left(\frac{(1 - |x|^{2})(1 - |y|^{2})}{(1 + |x|^{2})(1 + |y|^{2})}\right)^{n-1}
$$
and applying Euler’s transformation (see~\cite[p.~64, Eq.~(23)]{BatemanHTF}):
$$
{}_2F_1\!\left(
\begin{matrix}
a,\; b \\
c
\end{matrix}
;\, z
\right)
=
(1 - z)^{c - a - b}
\cdot
{}_2F_1\!\left(
\begin{matrix}
c - a,\; c - b \\
c
\end{matrix}
;\, z
\right),
$$
we obtain
\begin{equation}\label{koneccc}
\begin{aligned}
K_h(x,y)
&=
\left(
1+\frac{|x-y|^{2}}{(1-|x|^{2})(1-|y|^{2})}
\right)^{\,2-\tfrac{3n}{2}}
\sum_{p,l=0}^{n-1}
\frac{(1-n)_{p+l}}{p!\,l!}\,
\frac{(n-1)_p\,(n-1)_l}{\big(\tfrac{n}{2}\big)_{p+l}}
\left(\frac{|x|^{2}}{|x|^{2}-1}\right)^{p}
\left(\frac{|y|^{2}}{|y|^{2}-1}\right)^{l}
\\
&\quad\times
{}_2F_1\!\left(
\begin{matrix}
1+p-\tfrac{n}{2},\; 1+l-\tfrac{n}{2}
\\[2pt]
\tfrac{n}{2}+p+l
\end{matrix}
;\;
-\dfrac{|x-y|^{2}}{(1-|x|^{2})(1-|y|^{2})}
\right).
\end{aligned}
\end{equation}
Throughout the proof we may assume that $|x|,|y|<\tfrac14$.
Then the series in \eqref{kon} and \eqref{konec1} converge absolutely,
so rearrangements and the application of various transformations are justified.
The right-hand side of \eqref{koneccc} defines a holomorphic function on \( B^n \times B^n \) and agrees with \( K_h(x,y) \) on a nonempty open subset. By the identity principle (see~\cite[p.~210]{Rudin}), we conclude that \eqref{koneccc} holds for all \( x, y \in B^n \).
\end{proof}
\begin{remark}
    From Theorem~\ref{theorem necooo} we know that
    \begin{align*}
 K_h(x,y) 
&= 
\left(
    1 + 
    \frac{
        |x-y|^2
    }{
        (1 - |x|^{2})(1 - |y|^{2})
    }
\right)^{\!2 - \tfrac{3n}{2}}
\\[4pt]
&\quad \times
\sum_{p,l=0}^{\,n-1}
\frac{(1-n)_{p+l}}{p!\,l!}\,
\frac{(n-1)_p\,(n-1)_l}{\left(\tfrac{n}{2}\right)_{p+l}}
\left( \frac{|x|^{2}}{|x|^{2}-1} \right)^{\!p}
\left( \frac{|y|^{2}}{|y|^{2}-1} \right)^{\!l}
\\[4pt]
&\quad \times
{}_2F_1\!\left(
\begin{matrix}
1 + p - \tfrac{n}{2},\; 1 + l - \tfrac{n}{2} \\[4pt]
\tfrac{n}{2} + p + l
\end{matrix}
;\;
\frac{
    -|x-y|^2
}{
    (1 - |x|^{2})(1 - |y|^{2})
}
\right).
\end{align*}
\medskip
\noindent
Now, observe that when $n$ is even, the Gauss hypergeometric function
\begin{equation*}
{}_2F_1\!\left(
\begin{matrix}
1 + p - \tfrac{n}{2},\; 1 + l - \tfrac{n}{2} \\[4pt]
\tfrac{n}{2} + p + l
\end{matrix}
;\;
\frac{
    -|x-y|^2
}{
    (1 - |x|^{2})(1 - |y|^{2})
}
\right)
\end{equation*}
reduces to a polynomial whenever $p$ or $l < \tfrac{n}{2}$. 
In contrast, for $p, l \geq \tfrac{n}{2}$, the function ${}_2F_1$ admits the following representation (see~\cite[p.~110, Eq.~(10)]{BatemanHTF}):
\begin{equation*}\label{eq:derivative-identity-abc}
\begin{aligned}
{}_2F_1\!\left(
\begin{matrix}
a + 1,\; a + b + 1 \\[4pt]
a + b + c + 2
\end{matrix}
;\; z
\right)
&= 
\frac{(a + b + c + 1)!}{c!\, a!\, (a + b)!\, (b + c)!}\;
(-1)^{\,1 + b}\;
\frac{d^{\,a + b}}{dz^{\,a + b}}
\\[8pt]
&\quad \times 
\left\{
(1 - z)^{\,b + c}\;
\frac{d^{\,c}}{dz^{\,c}}
\left[
z^{-1}\log(1 - z)
\right]
\right\}.
\end{aligned}
\end{equation*}
Consequently, in even dimensions the kernel $K_h(x,y)$ admits a closed-form representation. This is in agreement with the results established by Stoll (see~\cite[Theorem~3.1]{Stoll2019}).
\end{remark}
\section{Szeg\"o kernel for linearly dependent arguments}
We now recall a proposition that will assist in analysing the Szegő kernel in the case \(y = \lambda x\).
\begin{proposition} \label{4.1} (see~\cite[Eq.~4.3]{ChoiRathie2013}).\label{thm:hypergeometric_identity}
  Let \(a,b,c \geq 0\) be parameters $c + \frac{1}{2}, \: 2c \notin \mathbb{Z}_{\leq 0}$, and let
  \(x,z\in\mathbb{R}\) satisfy
    $\:0 \le x < \tfrac14,\: 0 \le z < \tfrac14.$
  Then
  \begin{equation}\label{eq:hypergeometric_identity}
    (1+2x)^{-a}\,(1+2z)^{-b}\,F_1\!\left(
\begin{matrix}
c;\; a,\; b \\
2c
\end{matrix}
;\; \tfrac{4x}{1+2x},\,\tfrac{4z}{1+2z}
\right)
    \;=\;
    X_{9}\!\Bigl(a,\,b;\,c+\tfrac12;\,x^{2},\,2xz,\,z^{2}\Bigr).
  \end{equation}
\end{proposition}
  \begin{corollary}\label{prop:szego_radial}
      Let \(y\in B^n\) and \(0\leq \lambda \leq 1\).
      Then for \(x=\lambda y\) the Szegö kernel admits the representation
  $$
    K_h(x,y)
    =\frac{(1-|x|^2)^{n-1}\,(1-|y|^2)^{n-1}}
          {(1+|x|)^{2(n-1)}\,(1+|y|)^{2(n-1)}}
     \,
     F_1\!\left(
\begin{matrix}
\frac{n - 1}{2};\; n - 1,\; n - 1 \\
n - 1
\end{matrix}
;\; \frac{4|x|}{(1 + |x|)^2},\; \frac{4|y|}{(1 + |y|)^2}
\right).
$$
  \end{corollary}
\begin{proof}
 Direct consequence of Theorem 2.2 and Proposition \ref{thm:hypergeometric_identity}.
\end{proof}

\begin{remark}
When \( x = y \), the Appell function \( F_1 \) simplifies according to the classical identity (see~\cite[p.~239, Eq.~(11)]{BatemanHTF}):
\begin{equation}\label{jkl}
F_1\!\left(
    \begin{matrix}
        \alpha;\; \beta,\; \beta' \\[2pt]
        \gamma
    \end{matrix}
    ;\;
    z, z
\right)
=
{}_2F_1\!\left(
    \begin{matrix}
        \alpha,\; \beta + \beta' \\[2pt]
        \gamma
    \end{matrix}
    ;\;
    z
\right).
\end{equation}
Consequently, the diagonal value of the Szeg\H{o} kernel simplifies to
$$
K_h(x, x)
= \frac{(1 - |x|^2)^{2(n - 1)}}{(1 + |x|)^{4(n - 1)}}
\, {}_2F_1\!\left(
    \begin{matrix}
        \tfrac{n - 1}{2},\; 2(n - 1) \\[2pt]
        n - 1
    \end{matrix}
    ;\;
    \dfrac{4|x|}{(1 + |x|)^2}
\right).
$$
To simplify further we employ the following quadratic transformation~\cite[p.~64, Eq.~(24)]{BatemanHTF}:
$$
{}_2F_1\!\left(
\begin{matrix}
a,\; b \\
2a
\end{matrix}
;\; \frac{4t}{(1+t)^2}
\right)
=
(1+t)^{2b}\;
{}_2F_1\!\left(
\begin{matrix}
b,\; b-a+\tfrac12 \\
a+\tfrac12
\end{matrix}
;\; t^2
\right),
$$
which yields
$$
{}_2F_1\!\left(
\begin{matrix}
\frac{n-1}{2},\; 2(n-1) \\
n-1
\end{matrix}
;\; \frac{4|x|}{(1+|x|)^2}
\right)
=
(1+|x|)^{4(n-1)}
\;
{}_2F_1\!\left(
\begin{matrix}
2(n-1),\; \frac{3n-2}{2} \\
\frac{n}{2}
\end{matrix}
;\; |x|^2
\right).
$$
Subsequently, applying the 
we arrive at
$$
K_h(x,x) = (1-|x|^2)^{-(n-1)}
\,{}_2F_1\!\left( \begin{matrix} 2-\frac{3n}{2},\; -(n-1) \\[2pt] \frac{n}{2} \end{matrix} ;\; |x|^2 \right).
$$
This is in precise accordance with result established by Stoll (see~\cite[Theorem 4.1]{Stoll2019}).\\
This result can also be derived from Theorem~\ref{theorem necooo}. 
When $x = y$,~\eqref{szego_roz} reduces to 
\begin{equation*}
    K_h(x,x) 
    = 
    F_{1}\!\left(
    \begin{matrix}
    1 - n;\; n - 1,\, n - 1 \\[4pt]
    \dfrac{n}{2}
    \end{matrix}
    ;\;
    \dfrac{|x|^{2}}{1 - |x|^{2}},\;
    \dfrac{|x|^{2}}{1 - |x|^{2}}
    \right).
\end{equation*}
Applying the transformation formula~\cite[p.~239, Eq.~(1)]{BatemanHTF}:
\begin{equation*}
F_{1}\!\left(
\begin{matrix}
\gamma - \alpha;\; \beta,\, \beta' \\[4pt]
\gamma
\end{matrix}
;\;
\dfrac{x}{\,x - 1\,},\;
\dfrac{y}{\,y - 1\,}
\right) 
= 
(1 - x)^{\beta}\,(1 - y)^{\beta'}\,
F_{1}\!\left(
\begin{matrix}
\alpha;\; \beta,\, \beta' \\[4pt]
\gamma
\end{matrix}
;\; x,\, y
\right),
\end{equation*}
we obtain
\begin{equation*}
\small
F_{1}\!\left(
\begin{matrix}
1 - n;\; n - 1,\, n - 1 \\[4pt]
\dfrac{n}{2}
\end{matrix}
;\;
\dfrac{|x|^{2}}{|x|^{2} - 1},\;
\dfrac{|x|^{2}}{|x|^{2} - 1}
\right)
=
\bigl(1-|x|^{2}\bigr)^{2(n - 1)}\,
F_{1}\!\left(
\begin{matrix}
\dfrac{3n}{2} - 1;\; n - 1,\, n - 1 \\[4pt]
\dfrac{n}{2}
\end{matrix}
;\;
|x|^{2},\;
|x|^{2}
\right).
\end{equation*}
Finally, the result follows by applying~\eqref{jkl}, followed by the Euler transformation.
\end{remark}
\section{Bergman reproducing kernel}
We now provide a series expansion for the weighted Bergman kernel.
\begin{theorem} \label{teylorbergman}
Let \( x, y \in B^n \). Then the following series expansion holds:
\begin{align*}
K_h^{s}(x, y)
&= 
\frac{(1 - |x|^2)^{n - 1} (1 - |y|^2)^{n - 1}}
     {(1 + |x|^2)^{n - 1} (1 + |y|^2)^{n - 1}}
\sum_{\gamma = 0}^\infty
\sum_{\alpha = 0}^\infty
\sum_{\beta = 0}^\infty
\mathcal{A}_{\alpha,\,\beta,\,\gamma}\,
X^{\alpha}\,
Y^{\gamma}\,
Z^{\beta}.
\end{align*}
 where
$$
    X = \left(\frac{\lvert x\rvert}{1+\lvert x\rvert^2}\right)^2,\quad
    Y = \frac{2\langle x,y\rangle}{(1+\lvert x\rvert^2)\,(1+\lvert y\rvert^2)},\quad
    Z = \left(\frac{\lvert y\rvert}{1+\lvert y\rvert^2}\right)^2
$$
and 
\begin{align*}
\mathcal{A_{\alpha, \,\beta, \,\gamma}} &:= \sum_{j=0}^{\min(\alpha,\beta)}
\frac{(-1)^j\big(\frac{n-2}{2}\big)_{\gamma+j}\,[n-2+2(\gamma+2j)]}
     {I_{\gamma+2j}(s)\,(n-2)\,j!\,\gamma!}\\
&\quad\times
\frac{(n-1)_{2\alpha+\gamma}(-\alpha)_j}{(\frac{n}{2})_{\alpha+\gamma+j}\,\alpha!}
\frac{(n-1)_{2\beta+\gamma}(-\beta)_j}{(\frac{n}{2})_{\beta+\gamma+j}\,\beta!}.
\end{align*}
\end{theorem}

\begin{proof}
    From~\cite[Theorem~5.3]{Stoll2019} it follows that
\begin{equation}\label{vazenejadro}
    K_h^{s}(x,y)
    =
    \sum_{m = 0}^{\infty}
    \frac{
        S_m(|x|^{2})\, S_m(|y|^{2})\, Z_m(x,y)
    }{
        I_m(s)
    },
\end{equation}
where \(I_m(s)\) is given by
\begin{equation*}
    I_m(s)
    :=
    \frac{
        \Gamma\!\left( \tfrac{n}{2} + s + 1 \right)
    }{
        \Gamma\!\left( \tfrac{n}{2} \right)
        \Gamma(s + 1)
    }
    \int_{0}^{1}
        t^{\,m + \tfrac{n}{2} - 1}
        (1 - t)^{s}
        \, S_m(t)^{2}
    \, dt,
\end{equation*}
and \(S_m(t)\) is defined by
\begin{equation*}
    S_m(t)
    :=
    \frac{(n - 1)_{m}}{ \bigl( \tfrac{n}{2} \bigr)_{m} }\;
    {}_2F_1\!\left(
        \begin{matrix}
            m,\; 1 - \tfrac{n}{2} \\[4pt]
            m + \tfrac{n}{2}
        \end{matrix}
        ;\;
        t
    \right).
\end{equation*}
Here \(Z_m(x,y)\) denotes the zonal harmonic of degree~\(m\).\\
From the homogeneity of $Z_m(x,y)$ and \eqref{zonal} we have
    \begin{align*}
         Z_m(x,y) &= \frac{n-2+2m}{n-2}|x|^m|y|^mC_m^{\frac{n-2}{2}}\Big(\frac{\langle x,y\rangle}{|x||y|}\Big)\\[1.0ex]
         &= \frac{n-2+2m}{n-2}|x|^m|y|^m\sum_{j=0}^{\lfloor m/2\rfloor}
\frac{(-1)^j\,\Gamma(m-j+\frac{n-2}{2})}{j!\,(m-2j)!\,\Gamma(\frac{n-2}{2})}
\,\bigg(\frac{2\langle x,y\rangle}{|x||y|}\Bigg)^{m-2j}.
    \end{align*}
To simplify \(S_m(t)\), we first apply the quadratic transformation (see~\cite[p.~64]{BatemanHTF}):
$$
{}_2F_1\!\left(
\begin{matrix}
a,\; b \\[2pt]
a - b + 1
\end{matrix}
;\, z
\right)
=
(1 + z)^{-a}
\cdot
{}_2F_1\!\left(
\begin{matrix}
\tfrac{1}{2}a,\; \tfrac{1}{2}a + \tfrac{1}{2} \\[2pt]
a - b + 1
\end{matrix}
;\,
\frac{4z}{(1 + z)^2}
\right).
$$
Thus, we obtain
$$
S_m(t)
=
\frac{(n - 1)_m}{\left( \tfrac{n}{2} \right)_m}
\cdot
(1 + t)^{-m}
\cdot
{}_2F_1\!\left(
\begin{matrix}
\tfrac{m}{2},\; \tfrac{m}{2} + \tfrac{1}{2} \\
m + \tfrac{n}{2}
\end{matrix}
;\,
\frac{4t}{(1 + t)^2}
\right).
$$
Next, we apply Euler's transformation and see that \(S_m(t)\) can be written as
$$
S_m(t)
=
\frac{(n - 1)_{m}}{\left( \tfrac{n}{2} \right)_{m}}
(1 + t)^{-m}
\left( \frac{1 - t}{1 + t} \right)^{n - 1}
{}_2F_1\!\left(
\begin{matrix}
\tfrac{m}{2} + \tfrac{n}{2},\; \tfrac{m}{2} + \tfrac{n - 1}{2} \\[4pt]
m + \tfrac{n}{2}
\end{matrix}
;\;
\dfrac{4t}{(1 + t)^{2}}
\right).
$$
Using this to expand \eqref{vazenejadro} yields in:
\begin{align*}
K_h^{s}(x,y)
&=
\frac{(1 - |x|^2)^{n - 1} (1 - |y|^2)^{n - 1}}{(1 + |x|^2)^{n - 1} (1 + |y|^2)^{n - 1}}
\sum_{m=0}^\infty
\sum_{k=0}^\infty
\sum_{l=0}^\infty
\sum_{j=0}^{\lfloor m/2 \rfloor}
\frac{1}{I_m(s)}
\cdot
\left[ \frac{(n - 1)_m}{\left( \tfrac{n}{2} \right)_m} \right]^2
\\[1ex]
&\quad\times
\frac{\left( \tfrac{m}{2} + \tfrac{n}{2} \right)_k \left( \tfrac{m}{2} + \tfrac{n - 1}{2} \right)_k}{\left( m + \tfrac{n}{2} \right)_k\,k!}
\left( \frac{4|x|^2}{(1 + |x|^2)^2} \right)^k
(1 + |x|^2)^{-m}
\\[1ex]
&\quad\times
\frac{\left( \tfrac{m}{2} + \tfrac{n}{2} \right)_l \left( \tfrac{m}{2} + \tfrac{n - 1}{2} \right)_l}{\left( m + \tfrac{n}{2} \right)_l\,l!}
\left( \frac{4|y|^2}{(1 + |y|^2)^2} \right)^l
(1 + |y|^2)^{-m}
\\[1ex]
&\quad\times
\frac{n - 2 + 2m}{n - 2}
\cdot |x|^{m} |y|^{m}
\cdot \frac{(-1)^j\,\Gamma\left( m - j + \tfrac{n - 2}{2} \right)}{j!\,(m - 2j)!\,\Gamma\left( \tfrac{n - 2}{2} \right)}
\cdot \bigg(\frac{ 2\langle x, y \rangle }{|x||y|}\bigg)^{m - 2j}.\\
\end{align*}
Using the identities
$$
\left( \tfrac{m}{2} + \tfrac{n}{2} \right)_k \left( \tfrac{m}{2} + \tfrac{n - 1}{2} \right)_k \cdot 4^k = ( m +n - 1)_{2k}, \qquad
\left( m + \tfrac{n}{2} \right)_m \left( \tfrac{n}{2} \right)_k = \left( \tfrac{n}{2} \right)_{m + k},
$$
together with
$$
(m + n - 1)_{2k} \cdot (n - 1)_m = (n - 1)_{2k + m},\quad\quad\frac{\Gamma(m-j+\frac{n-1}{2})}{\Gamma(\frac{n-1}{2})} = \bigg(\frac{n-1}{2}\bigg)_{m-j},
$$
and analogously for the index $l$, we obtain:
\begin{align*}    
K_h^{s}(x,y)
&=
\frac{(1 - |x|^2)^{n - 1} (1 - |y|^2)^{n - 1}}{(1 + |x|^2)^{n - 1} (1 + |y|^2)^{n - 1}}
\sum_{m=0}^\infty
\sum_{k=0}^\infty
\sum_{l=0}^\infty
\sum_{j=0}^{\lfloor m/2 \rfloor}
\frac{1}{I_m(s)}
\cdot
\\[1ex]
&\quad\times
\frac{\left(n - 1 \right)_{2k+m}}{\left(\tfrac{n}{2} \right)_{k+m}\,k!}
|x|^{2k+2j}
(1 + |x|^2)^{-2k-m}
\\[1ex]
&\quad\times
\frac{\left(n -1 \right)_{2l+m} }{\left(\tfrac{n}{2} \right)_{l+m}\,l!}
|y|^{2l+2j}
(1 + |y|^2)^{-2l-m}
\\[1ex]
&\quad\times
\frac{n - 2 + 2m}{n - 2}
\cdot \frac{(-1)^j\,(\frac{n-2}{2})_{m-j}}{j!\,(m - 2j)!\,}
\cdot \big(2\langle x, y \rangle\big)^{m - 2j}.
\end{align*}
Now setting $\gamma = m - 2j$ and summing over $\gamma$ yields:
\begin{align*}
K_h^{s}(x,y)
&=\frac{(1 - |x|^2)^{n - 1} (1 - |y|^2)^{n - 1}}{(1 + |x|^2)^{n - 1} (1 + |y|^2)^{n - 1}}
\sum_{\gamma = 0}^\infty
\sum_{j = 0}^\infty
\sum_{k = 0}^\infty
\sum_{l = 0}^\infty
\frac{1}{I_{\gamma + 2j}(s)}
\\[1ex]
&\quad\times
\frac{(n - 1)_{2k + \gamma + 2j}}{\left( \tfrac{n}{2} \right)_{k + \gamma + 2j} \, k!}
\bigg(\frac{|x|}{(1 + |x|^2)}\bigg)^{2k + 2j}
\\[1ex]
&\quad\times
\frac{(n - 1)_{2l + \gamma + 2j}}{\left( \tfrac{n}{2} \right)_{l + \gamma + 2j} \, l!}
\bigg(\frac{|y|}{(1 + |y|^2)}\bigg)^{2k + 2j}
\\[1ex]
&\quad\times
\frac{n - 2 + 2(\gamma + 2j)}{n - 2}
\cdot
\frac{(-1)^j \left( \tfrac{n - 2}{2} \right)_{\gamma + j}}{j! \cdot \gamma!}
\cdot \bigg(\frac{ 2 \langle x, y \rangle}{(1 + |x|^2)(1 + |y|^2)}\bigg)^{\gamma}.
\end{align*}
Now making a substitution \(\alpha = k + j\) and \(\beta = l + j\), the $K_h^{s}(x, y)$ becomes:

\begin{align}\label{finaldestination}
&\frac{(1 - |x|^2)^{n - 1} (1 - |y|^2)^{n - 1}}{(1 + |x|^2)^{n - 1} (1 + |y|^2)^{n - 1}}
\sum_{\gamma = 0}^\infty
\sum_{\alpha = 0}^\infty
\sum_{\beta = 0}^\infty
\sum_{j = 0}^{\min(\alpha,\, \beta)}
\frac{1}{I_{\gamma + 2j}(s)}
\\[1ex]\notag
&\quad\times
\frac{(n - 1)_{2\alpha + \gamma}}{\left( \tfrac{n}{2} \right)_{\alpha + \gamma + j} \, (\alpha - j)!}
\cdot \left( \frac{|x|}{1 + |x|^2} \right)^{2\alpha}
\\[1ex]\notag
&\quad\times
\frac{(n - 1)_{2\beta + \gamma}}{\left( \tfrac{n}{2} \right)_{\beta + \gamma + j} \, (\beta - j)!}
\cdot \left( \frac{|y|}{1 + |y|^2} \right)^{2\beta}
\\[1ex]\notag
&\quad\times
\frac{n - 2 + 2(\gamma + 2j)}{n - 2}
\cdot
\frac{(-1)^j \left( \tfrac{n - 2}{2} \right)_{\gamma + j}}{j! \cdot \gamma!}
\cdot \left( \frac{2 \langle x, y \rangle}{(1 + |x|^2)(1 + |y|^2)} \right)^{\gamma}.
\end{align}
Using the identity $(\alpha -j)! = \frac{(-1)^j}{(-\alpha)_j}\alpha!$ and similarly for $(\beta-j)!$ we can write
\begin{align*}
K_h^{s}(x,y)
&= \frac{(1 - |x|^2)^{n - 1} (1 - |y|^2)^{n - 1}}{(1 + |x|^2)^{n - 1} (1 + |y|^2)^{n - 1}}
\sum_{\gamma = 0}^\infty
\sum_{\alpha = 0}^\infty
\sum_{\beta = 0}^\infty
\sum_{j = 0}^{\min(\alpha,\, \beta)}
\frac{1}{I_{\gamma + 2j}(s)}
\\[1ex]
&\quad\times
\frac{(n - 1)_{2\alpha + \gamma}(-\alpha)_j}{\left( \tfrac{n}{2} \right)_{\alpha + \gamma + j} \, \alpha!}
\cdot \left( \frac{|x|}{1 + |x|^2} \right)^{2\alpha}
\\[1ex]
&\quad\times
\frac{(n - 1)_{2\beta + \gamma}(-\beta)_j}{\left( \tfrac{n}{2} \right)_{\beta + \gamma + j} \, \beta!}
\cdot \left( \frac{|y|}{1 + |y|^2} \right)^{2\beta}
\\[1ex]
&\quad\times
\frac{n - 2 + 2(\gamma + 2j)}{n - 2}
\cdot
\frac{(-1)^j \left( \tfrac{n - 2}{2} \right)_{\gamma + j}}{j! \cdot \gamma!}
\cdot \left( \frac{2 \langle x, y \rangle}{(1 + |x|^2)(1 + |y|^2)} \right)^{\gamma},
\end{align*}
as claimed.
\end{proof}
\begin{remark}
The Hardy space inner product corresponds to $I_m(s)=1$ $\forall m$ and $K_h^{s}(x,y)$ reduces to the Szeg\H{o} kernel $K_h(x,y)$.\\
To see this, observe that
\begin{align*}
K_h^{s}(x,y)
&= \frac{(1 - |x|^2)^{n-1} (1 - |y|^2)^{n-1}}
         {(1 + |x|^2)^{n-1} (1 + |y|^2)^{n-1}}
\sum_{\gamma=0}^\infty
\sum_{\alpha=0}^\infty
\sum_{\beta=0}^\infty
\sum_{j=0}^{\min(\alpha,\beta)}
\\[1ex]
&\qquad\times
\frac{(n-1)_{2\alpha+\gamma}\,(-\alpha)_j}
     {\bigl(\tfrac{n}{2}\bigr)_{\alpha+\gamma+j}\,\alpha!}
\left( \frac{|x|}{1 + |x|^2} \right)^{2\alpha}
\cdot
\frac{(n-1)_{2\beta+\gamma}\,(-\beta)_j}
     {\bigl(\tfrac{n}{2}\bigr)_{\beta+\gamma+j}\,\beta!}
\left( \frac{|y|}{1 + |y|^2} \right)^{2\beta}
\\[1ex]
&\qquad\times
\frac{n-2+2(\gamma+2j)}{n-2}\,
\frac{(-1)^j \bigl(\tfrac{n-2}{2}\bigr)_{\gamma+j}}
     {j!\,\gamma!}
\left( \frac{2\langle x, y\rangle}
            {(1 + |x|^2)(1 + |y|^2)} \right)^{\gamma}.
\end{align*}
Using
$$
n-2+2(\gamma+2j)
=(n-2+2\gamma)\,
\frac{
\Gamma\!\left(\tfrac{n-2+2\gamma}{4}+1+j\right)\Gamma\!\left(\tfrac{n-2+2\gamma}{4}\right)
}{
\Gamma\!\left(\tfrac{n-2+2\gamma}{4}+1\right)\Gamma\!\left(\tfrac{n-2+2\gamma}{4}+j\right)
}
=(n-2+2\gamma)\,
\frac{\bigl(\tfrac{n-2+2\gamma}{4}+1\bigr)_j}{\bigl(\tfrac{n-2+2\gamma}{4}\bigr)_j},
$$
we can rewrite
\begin{align*}
K_h^{s}(x,y)
&=
\frac{(1-|x|^2)^{n-1}(1-|y|^2)^{n-1}}{(1+|x|^2)^{n-1}(1+|y|^2)^{n-1}}
\sum_{\gamma=0}^\infty\sum_{\alpha=0}^\infty\sum_{\beta=0}^\infty\sum_{j=0}^{\min(\alpha,\beta)}
\\[2pt]
&\quad\times
\frac{(n-1)_{2\alpha+\gamma}\,(-\alpha)_j}{\bigl(\tfrac{n}{2}\bigr)_{\alpha+\gamma+j}\,\alpha!}
\left(\frac{|x|}{1+|x|^2}\right)^{2\alpha}
\frac{(n-1)_{2\beta+\gamma}\,(-\beta)_j}{\bigl(\tfrac{n}{2}\bigr)_{\beta+\gamma+j}\,\beta!}
\left(\frac{|y|}{1+|y|^2}\right)^{2\beta}
\\[2pt]
&\quad\times
\frac{n-2+2\gamma}{n-2}\,
\frac{\bigl(\tfrac{n-2+2\gamma}{4}+1\bigr)_j}{\bigl(\tfrac{n-2+2\gamma}{4}\bigr)_j}\,
\frac{(-1)^j\bigl(\tfrac{n-2}{2}\bigr)_{\gamma+j}}{j!\,\gamma!}
\left(\frac{2\langle x,y\rangle}{(1+|x|^2)(1+|y|^2)}\right)^{\gamma}
\\[6pt]
&=
\frac{(1-|x|^2)^{n-1}(1-|y|^2)^{n-1}}{(1+|x|^2)^{n-1}(1+|y|^2)^{n-1}}\\
&\quad\times\sum_{\gamma,\alpha,\beta}^\infty
\frac{(n-1)_{2\alpha+\gamma}(n-1)_{2\beta+\gamma}}{\alpha!\,\beta!\,\gamma!}
\left(\frac{|x|}{1+|x|^2}\right)^{2\alpha}
\left(\frac{|y|}{1+|y|^2}\right)^{2\beta}
\left(\frac{2\langle x,y\rangle}{(1+|x|^2)(1+|y|^2)}\right)^{\gamma}
\\[2pt]
&\quad\times
\frac{n-2+2\gamma}{n-2}\,
\frac{\bigl(\tfrac{n-2}{2}\bigr)_{\gamma}}
     {\bigl(\tfrac{n}{2}\bigr)_{\alpha+\gamma}\bigl(\tfrac{n}{2}\bigr)_{\beta+\gamma}}\;
{}_4F_3\!\left(
\begin{matrix}
-\alpha,\; -\beta,\; \tfrac{n-2}{2}+\gamma,\; \tfrac{n-2+2\gamma}{4}+1 \\[2pt]
\tfrac{n}{2}+\alpha+\gamma,\; \tfrac{n}{2}+\beta+\gamma,\; \tfrac{n-2+2\gamma}{4}
\end{matrix}
;\,-1
\right).
\end{align*}
By \cite[p.~190, Eq.~(4)]{BatemanHTF}, 
$$
{}_4F_3\!\left(
\begin{matrix}
a,\; 1+\tfrac{a}{2},\; b,\; c \\[3pt]
\tfrac{a}{2},\; 1+a-b,\; 1+a-c
\end{matrix}
;\, -1
\right)
=
\frac{\Gamma(1+a-b)\,\Gamma(1+a-c)}
     {\Gamma(1+a)\,\Gamma(1+a-b-c)}.
$$
Applying this with \(a=\tfrac{n-2}{2}+\gamma\), \(b=-\alpha\), \(c=-\beta\), we obtain
$$
{}_4F_3\!\left(
\begin{matrix}
-\alpha,\; -\beta,\; \tfrac{n-2}{2}+\gamma,\; \tfrac{n-2+2\gamma}{4}+1 \\[3pt]
\tfrac{n}{2}+\alpha+\gamma,\; \tfrac{n}{2}+\beta+\gamma,\; \tfrac{n-2+2\gamma}{4}
\end{matrix}
;\, -1
\right)
=
\frac{
\Gamma\!\left(\tfrac{n}{2}+\alpha+\gamma\right)
\Gamma\!\left(\tfrac{n}{2}+\beta+\gamma\right)
}{
\Gamma\!\left(\tfrac{n}{2}+\gamma\right)
\Gamma\!\left(\tfrac{n}{2}+\alpha+\beta+\gamma\right)
}.
$$
Moreover, since
$$
\frac{\bigl(\tfrac{n}{2}-1\bigr)_{\gamma}}{\bigl(\tfrac{n}{2}\bigr)_{\gamma}}
= \frac{\tfrac{n}{2}-1}{\tfrac{n}{2}+\gamma-1},
$$
it follows that
\begin{align*}
K_h^{s}(x,y)
&=
\frac{(1-|x|^2)^{\,n-1}(1-|y|^2)^{\,n-1}}
     {(1+|x|^2)^{\,n-1}(1+|y|^2)^{\,n-1}}\\
&\times\sum_{\gamma,\alpha,\beta=0}^\infty
\frac{(n-1)_{2\alpha+\gamma}\,(n-1)_{2\beta+\gamma}}
     {\bigl(\tfrac{n}{2}\bigr)_{\alpha+\beta+\gamma}\,\alpha!\,\beta!\,\gamma!}
\left(\frac{|x|}{1+|x|^2}\right)^{2\alpha}
\left(\frac{|y|}{1+|y|^2}\right)^{2\beta}
\left(\frac{2\langle x,y\rangle}{(1+|x|^2)(1+|y|^2)}\right)^{\gamma}.
\\[2pt]
\end{align*}
By Corollary~\ref{cor:szego_kernel}, $K_h^{s}(x,y)=K_h(x,y)$, as expected.
\end{remark}

Here I would like to thank my supervisor Miroslav Engliš for introducing me to this topic and for his many valuable remarks, including Remark 5.2. I would also like to express my gratitude to Petr Blaschke for his help.\\Research was supported by GA CR grant no. 25-18042S.

\end{document}